\documentclass[amssymb, 11pt]{amsart}
\usepackage{latexsym}
\usepackage{young}

\newlength{\standardunitlength}
\newcommand{\bea}{\begin{eqnarray}}
\newcommand{\ena}{\end{eqnarray}}
\newcommand{\beas}{\begin{eqnarray*}}
\newcommand{\enas}{\end{eqnarray*}}

\newcommand{\ignore}[1]{}

\def\O{\Omega}

\def\e{\epsilon}
\def\l{\lambda}
\def\t{\theta}
\newcommand{\GL}{\operatorname{GL}}

\newcommand{\SL}{\operatorname{SL}}

\newcommand{\F}{\mathbb{F}}
\newcommand{\CC}{\mathbb{C}}
\newcommand{\ZZ}{\mathbb{Z}}

\newtheorem{prop}{Proposition}[section]

\newtheorem{lemma}[prop]{Lemma}
\newtheorem{cor}[prop]{Corollary}
\newtheorem{theorem}[prop]{Theorem}

\begin{document}

\title[Permutation Representations of Nonsplit Extensions]{Permutation Representations of Nonsplit Extensions involving  Alternating Groups}



\author{Robert M. Guralnick}
\address{Department of Mathematics\\University of Southern California\\
        Los Angeles, CA 90089, USA}
\email{guralnic@usc.edu}

\author{Martin W. Liebeck}
\address{Department of Mathematics, Imperial College,  London SW7 2AZ, UK}
\email{m.liebeck@imperial.ac.uk}

\thanks{Guralnick was partially supported by NSF grant DMS-1600056.  This work was started during
the conference Midrasha 20 and the authors wish to thank the Institute for Advanced Studies at Hebrew
University for its support.   We thank L\'aszl\'o Babai for discussing the problem with us.
 We also thank Eamonn O'Brien for help with some MAGMA computations. }


\date{\today}

\begin{abstract}    L. Babai has shown that a faithful permutation representation of a nonsplit extension of a group by an alternating group $A_k$ must have degree at least $k^2(\frac{1}{2}-o(1))$, and has asked how sharp this lower bound is. We prove that Babai's bound is sharp (up to a constant factor), by showing that 
there are such nonsplit extensions that have faithful permutation representations
of degree $\frac{3}{2}k(k-1)$. We also reprove Babai's quadratic lower bound with the constant $\frac{1}{2}$ improved to 1 (by completely different methods).
\end{abstract}

\maketitle

\begin{center} {\it Dedicated to our friend and colleague Alex Lubotzky} \end{center}

\section{Introduction}

Let $A_k$ and $S_k$ denote the alternating and symmetric groups of degree $k$.  
We consider finite group extensions $H$ of the form  
\begin{equation}\label{exten}
1 \rightarrow M \rightarrow H \rightarrow A_k \rightarrow 1,
\end{equation}
where $M \ne 1$ and the extension is nonsplit.

In a November 2016 lecture at the Jerusalem conference ``60 Faces to Groups" in honour of Alex Lubotzky's 60th birthday, Laci Babai discussed faithful permutation representations of such groups $H$, and 
noted that he had proved a lower bound of the form $k^2(\frac{1}{2}-o(1))$ for the degree of such a representation; this bound appears in \cite{Ba}. He asked how close to best possible his bound is, suggesting (perhaps provocatively) that there might be an exponential lower bound for the degree of the form $C^k$ for some constant $C>1$.  

In this note, we show that Babai's lower bound is in fact sharp (up to a constant factor), and we also give a different proof of his quadratic lower bound. Here are our two main results.

\begin{theorem} \label{lower}   Let $k > 20$.   If $H$ is a nonsplit extension as in $(\ref{exten})$, and $H$ embeds in
$S_{\ell}$ for some $\ell$, then $\ell \ge k(k-1)$.
\end{theorem} 

With significantly more work, it is likely that the lower bound $k(k-1)$ in this result could be replaced by $\frac{3}{2}k(k-1)$. 
This would be best possible, by the next result (taking $p=3$ in part (i)).  

\begin{theorem} \label{upper}   Let $k \ge 10$ and let $p$ be a prime.
\begin{itemize} 
\item[{\rm (i)}] If $p$ is odd and $p$ divides $k$,  then there is a nonsplit extension $H$ as in $(\ref{exten})$, with $M$ an
elementary abelian $p$-group, such that $H$ has a faithful permutation representation of degree $\frac{1}{2}pk(k-1)$.
\item[{\rm (ii)}]  There is a nonsplit extension $H$ as in $(\ref{exten})$, with $M$ an elementary abelian
$2$-group, such that $H$ has a faithful transitive permutation 
representation of degree $2k(k-1)$.
\end{itemize}
\end{theorem}

In \S \ref{lowerbound}  we prove Theorem \ref{lower}, and in   \S \ref{example1}  and \S \ref{example2} we prove the two parts of
 Theorem \ref{upper}.   

Throughout we shall use the notation of \cite{J} for modules for symmetric groups $S_n$: for a field $F$ and a partition $\l$ of $n$, $S^\l$ denotes the Specht module, and $D^\l$ the irreducible module for $S_n$ over $F$ corresponding to $\l$. Also, if $H$ is a subgroup of $G$, and $V$ is an $FH$-module, then $V_H^G$ denotes the corresponding induced module for $G$.

\section{Lower bounds: proof of Theorem \ref{lower}}  \label{lowerbound} 

In this section we prove Theorem \ref{lower}.   Let $P(H)$ denote the minimal degree of a faithful permutation representation of a group $H$.

\begin{theorem}    Let $k>20$, and suppose $H$ is a nonsplit extension 
\[
  1 \rightarrow M \rightarrow H \rightarrow A_k \rightarrow 1.
\]
Then $P(H) \ge k(k-1)$.
\end{theorem}

\begin{proof}   Suppose false, and let $H$ be a minimal counterexample (for a fixed $k$).
Write $G=A_k$, and let $\pi$ be the projection map $H \rightarrow G$. If there is a proper subgroup $H_1$ of $H$ such that $\pi(H_1) = G$, then $H_1$ is still a nonsplit extension and so $P(H_1)\ge k(k-1)$ by the minimality of $H$. This is a contradiction, since $P(H) \ge P(H_1)$. Hence there is no such $H_1$, and so $M \le \Phi(H)$, the Frattini subgroup. In particular, $M$  is nilpotent.  

As $P(H) < k(k-1)$, $H$ has a subgroup $H_0$ such that $[H:H_0]< k(k-1)$ and $H_0$ is maximal 
subject to not containing $M$.   If $H_0$ is not corefree in $H$, its core is a nontrivial normal subgroup $N$ of $H$; but then $H/N$ is a nonsplit extension of $M/N$ by $A_k$, and $P(H/N) \le [H/N:H_0/N] <k(k-1)$, contradicting the minimality of $H$. Hence $H_0$ is corefree.

We next claim that $\pi(H_0) = Y \cong A_{k-1}$. For if not, then $[A_k:\pi(H_0)] \ge \frac{1}{2}k(k-1)$ 
(see \cite[1.1]{L}), and so $[H:H_0]\ge k(k-1)$, a contradiction.
Any nontrivial module for $Y$ has dimension at least $k-3$ (see \cite[5.3.7]{KL}). Hence if $M$ has no trivial $Y$-quotient,  then $[H:H_0] \ge  k \cdot 2^{k-3}>k(k-1)$, a contradiction.
Thus  $Y$ acts trivially on $M/M_0$, where $M_0 = M\cap H_0$.   By the maximality of $H_0$, it follows 
that $M_0$ has prime index $p$ in $M$.   Since $M_0$ is corefree, this implies that $M$ is an elementary abelian
$p$-group and hence is an $\F_pG$-module.  

There is a  surjective $Y$-homomorphism from  $M$ to $\F_p$.   By Frobenius reciprocity and the fact
that $M_0$ is corefree, it follows that as an $\F_pG$-module,  $M$ injects into the induced module $(\F_{p})_Y^G$.  

Suppose first that $p$ does not divide $k$. Then
$(\F_{p})_Y^G \cong \F_p \oplus S^{(k-1,1)}$.   If $p \ne 2$, then $H^2(G, (\F_{p})_Y^G) \cong H^2(Y, \F_p) = 0$ and also  $H^2(G, \F_p) = 0$ (see \cite{KP}) whence $H^2(G,M)=0$, a contradiction.  
If $p=2$, then the same computation shows that  $H^2(G, (\F_{2})_Y^G)$ is $1$-dimensional
and similarly $H^2(Y, \F_2)$ is $1$-dimensional (again by \cite{KP}), 
whence $H^2(G,S^{(k-1,1)}) = 0$.   Thus,  $M$ is the trivial module and
so $H$ is the double cover of $G$.  However, the double cover of $A_{k-1}$ embeds in the double cover of $A_k$, whence
$\pi(H_0)$ is not $Y$, a contradiction.


Now suppose that $p$ divides $k$. Then $(\F_{p})_Y^G$ is a uniserial module with trivial socle and head, and heart equal to the 
irreducible module $D^{(k-1,1)}$ of dimension $k-2$ (notation of \cite[p.39]{J}).
If $p \ne 2$, it follows by \cite{KP} that $H^2(G,M)=0$ for any submodule $M$ of $(\F_{p})_Y^G$, a contradiction. 
Hence $p=2$. 
It is still true by \cite{KP} that $H^2(G, D^{(k-1,1)})=0$, whence $D^{(k-1,1)}$ is not a quotient of $M$.
Thus, either $M = \F_2$ or $M = (\F_{2})_Y^G$.  In the first case, we note as above that $\pi(H_0)$ cannot be $Y$, a contradiction. 
Hence $M = (\F_{2})_Y^G$, of dimension $k$.  Since $M$ is uniserial, it follows that
$H$ has a unique  minimal normal subgroup. 

Note that $[H:H_0]=2k$.   This implies that  there is a faithful irreducible $\CC H$-module  $W$ of dimension less than $2k$.
Now any $H$-orbit on the set of nontrivial linear characters of $M$ has size $1$, $k$ or at least $k(k-1)/2$. 
Since $W_M$ must have a linear 
constituent that is not fixed by $H$, by Clifford's theorem there are therefore precisely $k$ distinct linear
characters of $M$ occurring in $W$, and since $\dim W < 2k$, each occurs with multiplicity 1. Hence $\dim W = k$, and so 
$H$ embeds in $\GL_k(\CC)$.   Since $H$ is perfect  (it is perfect modulo the Frattini subgroup), in fact, 
$H$ embeds in $\SL_k(\CC)$. But the largest elementary abelian $2$-subgroup of $\SL_k(\CC)$ has rank $k-1$,
whereas $M \cong \F_2^k$, a contradiction.   This completes the proof. 
 \end{proof} 
 
 We can obtain a better lower bound in certain cases:
 
 \begin{theorem} \label{odd}    Let $k > 22$, and let $H$ be a nonsplit extension  
\[
1 \rightarrow M \rightarrow H \rightarrow A_k \rightarrow 1
\]
 such that $\gcd (2k, |M|)=1$.    
 Then $P(H) \ge \frac{1}{2}k(k-1)(k-2)$.
 \end{theorem}
 
 \begin{proof}     Suppose false, and let $H$ be a minimal counterexample. 
 We copy the previous proof. In particular, we deduce that $M\le \Phi(H)$, $M$ is nilpotent and $H$ has a corefree subgroup $H_0$ of index less than 
$\frac{1}{2}k(k-1)(k-2)$.  By \cite[1.1]{L},   $H_0M/M$ must contain $X:=A_{k-2}$.   As in the previous
 result,  $H_0$ must normalize a subgroup of prime index $r$ in $M$.   
Hence $M$ is an elementary abelian $r$-group for some odd prime $r$ with $r$ not dividing $k$.
   It follows that $M$ embeds in
$(\F_r)_X^{A_k}$.
In particular, the composition factors of $M$ are among  $D^{(k)}$, $D^{(k-1,1)}$, $D^{(k-2,2)}$  and $D^{(k-2,1,1)}$.
By  \cite[Thm. 2]{P} and
\cite[Thm 4.1, Prop. 5.4]{BKM}, it follows that $H^2(A_k, M)=0$, a contradiction. 
\end{proof}  

\noindent {\bf Remark } One can construct examples with $M$ a $3$-group such that $H$ has a faithful permutation representation of  degree $\frac{3}{2}k(k-1)(k-2)$.  The construction is very similar to those given in the next section.

\section{Proof of Theorem \ref{upper} -- the odd case}  \label{example1} 

In this section we prove Theorem \ref{upper}(i). 
Let $p$ be an odd prime and assume that $k \ge 10$ and $p$ divides $k$.      
Let $G=A_k$ and let $Y\cong S_{k-2}$ be a Young subgroup stabilizing a subset of size $2$. 

Let $S = S^{(k-2,1,1)}$ be the Specht module, and $D = D^{(k-2,1,1)}$ and $L = D^{(k-1,1)}$ 
be irreducible modules for $G$ over $\F_p$ (the restrictions to $G=A_k$ of the corresponding irreducibles for $S_k$, using the notation of \cite{J}).   We need the following relatively easy results about these modules.  The first follows from \cite[Theorem 2]{P}. 

\begin{lemma}  \label{wedge2}   
\begin{itemize}
\item[{\rm (i)}] $\dim S =\frac{1}{2} (k-1)(k-2)$, $\dim L = k-2$ and $\dim D = \frac{1}{2}(k-2)(k-3)$.

\item[{\rm (ii)}]  $S$ is indecomposable with socle isomorphic to $L$ and head isomorphic to $D$.
\end{itemize}
\end{lemma}

Let $\theta$ be the nontrivial 1-dimensional $\F_pY$-module.   

\begin{lemma}  \label{wedge2-cohomology}
\begin{itemize}
\item[{\rm (i)}] The induced module $\theta_Y^G$ has socle and head isomorphic to $L$, and the maximal submodule modulo the socle is isomorphic to $\F_p \oplus D$.
\item[{\rm (ii)}]  Let $M$ be the submodule of $\theta_Y^G$ with composition factors
$D$ and $L$.   Then
\begin{itemize}
\item[{\rm (a)}] $H^2(G,M) \ne 0$, and
\item[{\rm (b)}] the restriction  $M_Y \cong \theta  \oplus M_0$ for some $\F_pY$-module $M_0$.
\end{itemize}
\end{itemize}
\end{lemma}

\begin{proof} (i)  In characteristic $0$,  the Littlewood-Richardson rule shows that
$\theta_Y^G$ has composition factors  $S^{(k-2,1,1)}$ and $S^{(k-1,1)}$.
Hence using Lemma \ref{wedge2}(ii), we see that in characteristic $p$, the composition factors are $D$, $\F_p$ and $L$ (twice).

Let $V$ be the $\F_pG$-module $\t_Y^G$. We now compute the socle of $V$. By Frobenius Reciprocity, 
${\rm Hom}(\F_p,V) = 0$. Also $L_Y = \t \oplus L_0$  with $L_0$ irreducible, so ${\rm Hom}(L,V) \cong {\rm Hom}(L_Y,\t)$ is 1-dimensional.
Finally, $D \cong \wedge^2 L$, so $D_Y \cong (\t\otimes L_0) \oplus \wedge^2L_0$, and both summands are irreducible.
Hence ${\rm Hom}(D,V) \cong {\rm Hom}(D_Y, \t) = 0$. It follows that ${\rm soc}(V) \cong L$. Now the conclusion of (i) follows from the fact that $V$ is self-dual.

(ii) By the previous paragraph, $\t$ is the only 1-dimensional composition factor of $M_Y$, and it appears in the socle. On the other hand, $M \subseteq V = \t_Y^G$, so by Frobenius reciprocity $M_Y$ surjects onto $\theta$. Hence  $M_Y= \theta \oplus M_0$ as claimed.

It remains to show that $H^2(G,M) \ne 0$.    
First note that for $i=1,2$ we have 
\begin{equation}\label{coho}
H^i(G, V) = H^i(G,\theta_Y^G)=
H^i(Y, \theta)= 0
\end{equation}
 (here we are using the assumptions that $p$ is not 2 and $k > 9$).
Consider   
\[
0 \rightarrow M \rightarrow V \rightarrow  V/M \rightarrow 0.
\]
Then (\ref{coho}), together with the long exact sequence in cohomology, gives \linebreak
$H^2(G,M) \cong H^1(G,V/M)$. Since $V/M$ is isomorphic to the codimension 1 submodule of the $k$-dimensional permutation module 
for $G$ which is uniserial, $H^1(G,V/M) \ne 0$. This completes the proof. 
 \end{proof}
 
\noindent {\bf Proof of Theorem \ref{upper}(i) }
 Let $M$ be as in the previous lemma, and consider the group $H$ defined by a nonsplit extension as follows:
$$
1  \rightarrow M \rightarrow H \rightarrow A_k \rightarrow 1.
$$
(such a nonsplit extension exists, by Lemma \ref{wedge2-cohomology}(ii)(a)).
Since $M$ is uniserial and $H^2(G,L)=0$, $M$ is contained in the Frattini subgroup of $H$.  
Let $M_0$ the $Y$-invariant hyperplane of $M$, as in Lemma \ref{wedge2-cohomology}(ii)(b), and let $E = N_G(M_0)$.  
Then $E/M = Y$.   This gives
rise to the sequence
\[
1 \rightarrow M/M_0  \rightarrow  E/M_0 \rightarrow Y.
\]
As observed in (\ref{coho}), $H^2(Y, M/M_0) = H^2(Y,\t) = 0$, and so the above sequence splits.    
Thus $E$ contains a subgroup $E_0$ of index $p$. The action of $H$ on the cosets of $E_0$ 
maps $H$ into the symmetric group of degree $\frac{1}{2}pk(k-1)$.    
Since the core of $E_0$ is trivial,
this is an embedding of $H$.  This proves Theorem \ref{upper}(i).

  \section{Proof of Theorem \ref{upper} -- the even case}  \label{example2}  
  
  In this section we prove Theorem \ref{upper}(ii), the case where $p=2$.  In this case, unlike part (i) of the theorem, there is no restriction on residue of $k$ modulo $p$.
.    
 Let $k\ge 10$, let $G=A_k$ acting on $\Omega=\{1, \ldots, k\}$, and let $Y \cong S_{k-2}$ be
 the stabilizer  in $G$ of the subset $\{1, 2\}$. Write $F=\F_2$, and define $P = (F)_Y^G$, the permutation module over $F$ of $G$ acting on the set of pairs in $\O$. Define the fixed point space $C_P(Y) = \{v\in P: vy=v \ \forall y \in Y\}$.

 \begin{lemma}  \label{init}
 \begin{itemize}
 \item[{\rm (i)}]  The fixed point space $C_P(Y)$ has dimension $3$.
 \item[{\rm (ii)}]  $\dim H^2(G, P) = 2$.
 \end{itemize}
 \end{lemma}
 
 \begin{proof}  The first statement holds since the action of $G$ on pairs  has rank $3$.   The second follows by 
 the Eckmann--Shapiro Lemma \cite[Theorem 4]{E}: 
 $\dim H^2(G,P) = \dim H^2(G,(F)_Y^G ) = \dim H^2(Y, F) = 2$.
 \end{proof}

Let $e_{ij}$ denote a basis element of $P$ corresponding to the subset $\{i,j\}$ of $\O$, where $i < j$.   Note
 that this is an orthonormal basis with respect to the standard inner product   $(\,,\,)$ on $P$ (which is preserved by $G$).     
Now define the following elements of $P$: 
\[
\begin{array}{l}
x_i = \sum_j   e_{ij} \;\;(1\le i\le k) \\
f = \sum_{i,j}  e_{ij} \\
u = \sum_{2<i<j} e_{ij}.
\end{array}
\]
Note that a basis for the fixed point space $C_P(Y)$ is $\{u,f,y\}$, where $y=x_1 + x_2$. Since $P$ is self dual,
it follows also that $P/[Y,P]$ is $3$-dimensional, i.e. there is a $3$-dimensional trivial quotient of $P$ as an $FY$-module.

Assume now that $k \equiv 3 \pmod 4$. 
Then $P  = P_1 \oplus P_2 \oplus P_3  \cong F \oplus S^{(k-1,1)} \oplus S^{(k-2,2)}$ with each summand irreducible (where we identify each $S^{(k-i,i)}$ with its reduction modulo 2) . We can identify  
$P_1 = Ff$,  $P_2=Fx_1 + \ldots + Fx_k$ (of dimension $k-1$) and $P_3 =  (P_1 + P_2)^{\perp}$.     By Frobenius reciprocity,  $C_{P_i}(Y)$ has dimension 1 for each $i$.

Now $u$ is orthogonal to $P_1 + P_2$, since $(u,u)=0$ (as $k \equiv 3 \pmod 4$), whence
$(u,f)=0$  and $(u,x_i)= k -3 =0$ for $i > 2$ and $(u, x_i)=0$ for $i=1,2$.  Let $P_0 =(Fu)^{\perp}$.  Then $P_0$ is  a  $Y$-invariant hyperplane containing $P_1 + P_2$.  

Since $H^2(A_s, F)$ is $1$-dimensional for $s \ge 4$, it follows by Frobenius reciprocity that $\dim H^2(G, P_1)=1$. Also $H^2(G,P_2)=0$ by \cite{KP}, and so $\dim H^2(G,P_3)=1$ by Lemma \ref{init}(ii).   

Define the following two elements in $G$: 
\[
g_1 = (1 \  2)(3 \  4), \;g_2=(3 \  4)(5 \  6),
\]
so that $g_1 \in Y \cong S_{k-2}$ and $g_2 \in X$, where $X:=G_{12} \cong A_{k-2}$. 

We now consider $2$-cocycles.  We assume that all $2$-cocycles  $\delta$ are normalized so that $\delta(1,h)=\delta(h,1)=1$ for all $h$.

By the Eckmann-Shapiro Lemma, as we noted, we have an isomorphism from 
$H^2(Y, F)$ to $H^2(G,P)$ and these have dimension 2.  
For each of the four elements in $H^2(Y,F)$, we can choose a 2-cocycle $\e$ representing that element, and  
$\epsilon$ is completely determined (up to a coboundary) by $\epsilon(g_1, g_1)$ and $\epsilon(g_2,g_2)$.  

Let $\delta$ be a $2$-cocyle  representing  an element of $H^2(G, P)$ which is nontrivial in $H^2(G,P_3)$.  
Let $\epsilon \in H^2(Y, F)$ correspond to $\delta$ via the isomorphism given by the 
Eckmann-Shapiro Lemma.   Choose coset representatives $g_{ij}$ for $Y$ in $G$ as follows:
\[
\begin{array}{l}
g_{ij} = (1 \  i) (2 \  j),  \hbox{ if }2 < i < j,  \\
g_{1j} = (2 \ 1 \ j) \hbox{ if } j > 2, \\
g_{2j} = (1 \  2  \  j) \hbox{ if } j > 2, \\
g_{12} =1.
\end{array}
\]
Note that $g_{ij}$ sends $\{1,2\} \rightarrow \{i,j\}$ for all $i,j$.

We need some information about $\delta(g_i, g_i)$ for $i = 1, 2$. 

\begin{lemma}   \label{cocycle} 
\begin{itemize}  
\item[{\rm (i)}]    $\delta(g_1,g_1)$ is not contained in $P_0$.
\item[{\rm (ii)}]    $\delta(g_2, g_2)$ is contained in $P_0$.
\end{itemize}
\end{lemma}

\begin{proof}   We  proceed by induction on $k$ (assuming as above that $k \equiv 3 \pmod 4$). 
If $k = 7$ or $11$, the conclusion follows by direct computation using Magma \cite{M}.   So assume that $k \ge 15$.

Let  $\Omega_0 = \{1,\ldots ,k-4\}$, and let $G(\Omega_0) \cong A_{k-4}$ be the subgroup
of $G$ acting trivially on the complement of $\Omega_0$.  Note that  the permutation
module $P(\Omega_0)$  for $G(\O_0)$ acting on pairs in $\Omega_0$ is a $G(\Omega_0)$-summand of $P$.

Let $m_s = \delta(g_s,g_s)$ for $s=1,2$ and write $m_s = \sum  \alpha_{ij} e_{ij}$.  By the isomorphism given
in the Eckmann-Shapiro Lemma 
(cf.  \cite[p. 488]{E} or \cite[p. 43]{B}),  the determination of $\alpha_{ij}$ depends only on $\epsilon(g_s, g_s)$ and the coset representatives
$g_{ij}$ given above that take $\{1, 2 \}$ to $\{i, j \}$.   Note that if $i, j \in \Omega_0$, then $g_{ij}$ 
is in $G(\O_0$).      Thus, in computing $\alpha_{ij}$ for $i, j \in \Omega_0$, we can work in $G(\O_0)$. 
It follows that 
the projection of $\delta(g_s,g_s)$ in $P(\Omega_0)$ is precisely $\delta'(g_s,g_s)$, where $\delta'$
is the 2-cocycle corresponding to $\delta$, viewed as a function on  $G(\Omega_0) \times G(\Omega_0)$ with values
in $P(\Omega_0)$  (i.e. $\delta'$ corresponds to $\epsilon$  in the isomorphism given by the Eckmann-Shapiro Lemma
for the smaller group).  

Now define a collection of subsets of $\O$, as follows. Write $D = \{1,2, \ldots ,k-8\}$, and let 
\[
\begin{array}{l}
\O_0= D \cup \{k-7,k-6,k-5,k-4\}, \\
\O_1= D \cup \{k-3,k-2,k-1,k\}, \\
\O_2= D \cup \{k-7,k-6,k-3,k-2\}, \\
\O_3= D \cup \{k-5,k-4,k-3,k-2\}, \\
\O_4= D \cup \{k-7,k-6,k-1,k\}, \\
\O_5= D \cup \{k-5,k-4,k-1,k\}, \\
\O_6 = D.
\end{array}
\]
Then $u = \sum_0^6 u(\O_i)$, where $u(\O_i) = \sum_{2<r,s \in \O_i}e_{rs}$.

Let $m_j = \delta(g_j,g_j)$ for $j=1,2$. By induction we have $(m_1, u(\O_i)) = 1$, $(m_2, u(\O_i)) = 0$ for all $i$, and hence $(m_1,u)=1$ and $(m_2,u) = 0$. Both conclusions of the lemma follow. 
\end{proof}  

\begin{cor}\label{k3mod4}
Let $k\ge 7$ be an integer such that  $k \equiv 3 \pmod 4$, and let $M$ be the $\F_2A_k$-module that is the reduction modulo $2$ of the Specht module $S^{(k-2,2)}$.  Let $H$ be a nonsplit extension 
\[
1 \rightarrow M \rightarrow H \rightarrow A_k \rightarrow 1.
\]
Then $H$ has a faithful transitive permutation representation of degree $2k(k-1)$.
\end{cor}

\begin{proof}   Identify $M$ with $P/(P_1 + P_2)$, and define $M_0$ to be the $Y$-invariant
hyperplane $P_0/(P_1 + P_2)$. Let $\pi: H \rightarrow A_k$ be the canonical map with kernel $M$, and 
set $J = \pi^{-1}(Y)$ and $L= \pi^{-1}(X)$.  

We have $J = N_H(M_0)$.   Consider the group 
$L/M_0$.  By Lemma \ref{cocycle},  $L/M_0 \cong \ZZ/2 \times A_{k-2}$.   It follows that $H$ contains a subgroup
$L_0$ containing $M_0$ with $L_0/M_0 \cong A_{k-2}$.   Thus, $[H:L_0]=2k(k-1)$.   Since $M$ is the unique
minimal normal subgroup of $H$, it follows that $L_0$ is corefree in $H$.  Hence  
$H$ has a faithful transitive permutation representation of degree $2k(k-1)$.
\end{proof}

We can now prove Theorem \ref{upper}(ii). 

\begin{theorem}  Let $k \ge 7$.   Then there is a nonsplit extension 
\[
1 \rightarrow M \rightarrow H \rightarrow A_k \rightarrow 1,
\]
with $M$ an elementary abelian $2$-group, such that $H$ has a faithful transitive permutation representation of degree at most  $2k(k-1)$.
\end{theorem}

\begin{proof}  If $k \equiv 3 \pmod 4$, the conclusion follows from the previous result, so assume this is not the case.   
Write $k = j - i$, where $0 < i \le 3$ and $j \equiv 3 \pmod 4$.   Let 
\[
1 \rightarrow M \rightarrow H \rightarrow A_j \rightarrow 1.
\]
be the nonsplit  extension constructed in Corollary \ref{k3mod4}, and let $L_0$ be a corefree subgroup of $H$ of index $2j(j-1)$, 
with coset space $\Gamma = (H:L_0)$. Define $J$ to be the subgroup of $H$ containing $M$ with $J/M = A_k$.   

Observe that $J$ is a nonsplit extension of $M$ by $A_k$, since the coset $xM$ for $x = (1\,2)(3\,4) \in H/M = A_j$ 
consists of elements of order $4$ by Lemma \ref{cocycle}.   
 Also, the orbits of $J$ on $\Gamma$ have size $2$,  $2k$ or $2k(k-1)$.    Let   $J_0$ be a minimal subgroup
 of $J$  with $J = J_0M$.   Since $J$ is a nonsplit extension,  $N:=J_0 \cap M \ne 1$.   So $J_0$ is a nonsplit extension
 of $N$ by $A_k$.    Some orbit of $J$ must be nontrivial for $N$ and so the image of $J$ on this orbit must be nonsplit.
The conclusion follows.
 \end{proof}

\noindent {\bf Remark } It  is not hard to see that in fact $J_0$ is faithful only on the orbit of size $2k(k-1)$   (for $k$ sufficiently large,  this follows by Theorem \ref{lower}).

\end{document}